\documentclass[12pt,twoside]{article}
\usepackage{amsmath,amsbsy,amsfonts,amsthm,amssymb}
\usepackage[applemac]{inputenc}
\usepackage{color}
\newtheorem{theorem}{Theorem}[section]
 
 \newtheorem{lemma}[theorem]{Lemma}
 \newtheorem{proposition}[theorem]{Proposition}
 \theoremstyle{definition}
 
 \theoremstyle{remark}
 \newtheorem{remark}[theorem]{Remark}
 
 \newtheorem*{claim}{Claim} 
 \numberwithin{equation}{section}

\let\Section=\section
\def\section{\setcounter{equation}{0}\Section}

\title{Nodal solutions of a NLS equation concentrating on lower dimensional spheres}

\author{Giovany M. Figueiredo\thanks{Supported by PROCAD/CASADINHO 552101/2011-7, CNPq/PQ
301242/2011-9 and CNPq/CSF  200237/2012-8}\\
\noindent Faculdade de Matem\'{atica}\\
\noindent Universidade Federal do Par\'a \\
\noindent 66075-110, Bel\'{e}m - Pa, Brazil.\\
\noindent e-mail: {\tt{giovany@ufpa.br}}\\
\mbox{}\\
and \\
Marcos T. O. Pimenta\thanks{Corresponding author } \thanks{Supported by FAPESP 2014/16136-1 and CNPq 442520/2014-0}\\
\noindent Departamento de Matem\'atica e Computa\c{c}\~ao\\
\noindent Faculdade de Ci\^encias e Tecnologia\\
\noindent Universidade Estadual Paulista - Unesp\\
\noindent 19060-900, Presidente Prudente - SP, Brazil.\\
\noindent e-mail: {\tt{pimenta@fct.unesp.br}}\\
}
\date{}

\begin{document}

\maketitle

\begin{abstract}
In this work we deal with the following nonlinear Schr\"odinger equation
$$
\left\{
\begin{array}{l}-\epsilon^2\Delta u + V(x)u = f(u) \ \ \mbox{in
$\mathbb{R}^N$}\\
u \in H^1(\mathbb{R}^N),
\end{array}
\right.
$$
where $N \geq 3$, $f$ is a subcritical power-type nonlinearity and $V$ is a positive potential satisfying a local condition. We prove the existence and concentration of nodal solutions which concentrate around a $k-$ dimensional sphere of $\mathbb{R}^N$, where $1 \leq k \leq N-1$, as $\epsilon \to 0$. The radius of such sphere is related with the local minimum of a function which takes into account the potential $V$. Variational methods are used together with the penalization technique in order to overcome the lack of compactness.

\end{abstract}



\section{Introduction}

\hspace{.5cm} In the last decades, motivated by the great interest that this problem catch in quantum mechanics, so many researchers have dedicated their efforts on the study of the Nonlinear Schr\"odinger equation
$$
i\hbar \frac{\partial \psi}{\partial t} = -\frac{\hbar^2}{2m}\Delta \psi + W(x)\psi - |\psi|^{p-1}\psi \quad (t,x) \in \mathbb{R}\times \mathbb{R}^N.
$$ 
Of particular interest are the so called {\it standing wave solutions} which consists in solutions with a particle-like behavior. It is obtained by the {\it Ansatz} $\psi(t,x) = e^{-iEt/\hbar}u(x)$ which associate the NLS equation to its stationary version
\begin{equation}
-\epsilon^2\Delta u + V(x)u = |u|^{p-1}u \quad \mbox{in $\mathbb{R}^N$}
\label{IP1}
\end{equation}
where $\epsilon^2=\hbar^2/2m$ and $V(x) = W(x) - E$.
As far as (\ref{IP1}) is concerned, the behavior of the solutions when $\epsilon \to 0$ has a great physical interest since it describes the transition from quantum to the classical mechanics, being called {\it semiclassical states}. On this specific subject, many authors have worked on {\it spike-layered solutions} which are nontrivial ground-state points of the associated energy functional and tend to concentrate around one or more critical points of the potential $V$. We could cite some quite influent works on this subject, as the pioneering work of Floer and Weinstein \cite{Floer}, which have inspired the works of Rabinowitz \cite{Rabinowitz}, Wang \cite{Wang}, Del Pino and Felmer \cite{DelPino}, which have influenced so many other works by their own in the last three decades.

In the last ten years, solutions which concentrate on higher dimensional sets has been received more and more attention. The first work which seems to show this kind of result is \cite{Montenegro} in which the authors study a NLS equation on a bounded domain with Neumann boundary condition and prove the existence of a sequence of solutions which concentrate on some component of the boundary. One of the first works dealing with solutions concentrating around a sphere is \cite{Ambrosetti1} in which Ambrosetti, Malchiodi and Ni give necessary and sufficient conditions under which (\ref{IP1}) exhibit solutions concentrating around a sphere. The radius of such a sphere is given by a minimum point of a function $\mathcal{M}$, which takes into account the value of the radial potential $V(|x|)$. The role played by $\mathcal{M}$ is in order to balance the potential energy (coming from $V$) and the volume energy which arise from the other terms of the energy functional (see the introduction of \cite{Ambrosetti1} for more details). In fact, sphere concentrating solutions show a rather different behavior when compared with spike-layered ones. To be more specific, in \cite{Ambrosetti1}, the authors prove the existence of sphere concentrating solutions to (\ref{IP1}) even for critical or supercritical exponent $p$. This is in a strike contrast with the fact that, as showed in \cite{Cingolani}, no spike-layered solution exists to (\ref{IP1}) for $p = 2^*-1$. Other significant difference is that the energy of the sphere concentrating solutions tend to zero, in contrast with those of spike-layered solutions, whose energy converge to the mountain-pass level of the energy functional. In these and so many other works (\cite{Ambrosetti2,Ruiz,Daprile, Ianni} for example), Lyapunov-Schmidt reduction methods have been used in order to construct the sphere-concentrating solutions for Schr\"odinger equations, Schr\"odinger-Poisson systems other related problems.

More recently, in \cite{Bonheure1} Bonheure {\it et al} proved the existence of solutions concentrating around a $k-$dimensional sphere of $\mathbb{R}^N$, for all $k \in \{1,...,N-1\}$ to the following equation
\begin{equation}
-\epsilon^2\Delta u + V(x)u = K(x)f(u) \quad \mbox{in
$\mathbb{R}^N$,}
\label{IP2}
\end{equation}
where the potentials $V$ and $K$ satisfy rather generic conditions, allowing $V$ even to vanish on the infinity. To do so they use a modification of the penalization technique, originally presented in \cite{DelPino}, in such a way that compactness is recovered to the modified energy functional. Because of the generality of conditions under with $V$ and $K$ are subjected, in order to prove that the solutions of the modified problem are solutions of the original one, they made a thorough analysis with some barrier functions  which bounds the solutions from above. In \cite{Bonheure2} the authors employ a similar argument in order to show the existence of solutions concentrating on circumferences of $\mathbb{R}^3$, to a Schr\"odinger-Poisson system.

In the {\it spike-layered} solutions setting, the existence of sign-changing (or nodal) solutions was investigated by some authors. In \cite{Soares1} and \cite{Soares2}, Alves and Soares study problem (\ref{IP2}), with $K$ to be a constant, and prove the existence of nodal solutions which concentrate at minima of the potential $V$. In the first work they consider $f$ as a subcritical power-type nonlinearity and in the second, as presenting a critical exponential growth at infinity. In both they employ the penalization technique together with a careful analysis of the profile of the solutions. In \cite{Sato}, Y. Sato have proposed a different kind of penalization in order to show the existence of multi-peak nodal solutions to a Schr\"odinger equation with a vanishing potential.

A question that naturally arises is whether there exist a sequence of nodal solutions to the NLS equation which concentrate around a $k-$dimensional sphere. In this work we give a positive answer to this question. More specifically, we study the existence and concentration of nodal solutions to the following nonlinear Schr\"odinger equation

\begin{equation}
\left\{
\begin{array}{l}-\epsilon^2\Delta u + V(x)u = f(u) \ \ \mbox{in
$\mathbb{R}^N$}\\
u \in H^1(\mathbb{R}^N),
\end{array}
\right.  \label{P}
\end{equation}
where $N \geq 3$, exhibiting a cylindrical symmetry which implies in this sort of concentration. The nonlinearity $f$ is assumed to be a $C^1(\mathbb{R})$ odd function satisfying

\begin{itemize}
\item [$(f_1)$] There exists $\nu > 1$ such that $f(|s|) = o(|s|^\nu)$ as $s \to 0$;
\item [$(f_2)$] There exist $c_ 1, c_2 > 0$ such that $ |f'(s)| \leq c_1 + c_2|s|^{p-1}$ where $1 < p < \frac{N-k+2}{(N-k-2)}$ and $k$ is like in Section \ref{section1.1};
\item [$(f_3)$] There exists $\theta > 2 $ such that $$0 < \theta F(s) \leq f(s)s, \quad \mbox{for $s \neq 0$},$$
where $F(s) = \int_0^s f(t)dt$;
\item [$(f_4)$] $s \mapsto f(s)/s$ is increasing in $s > 0$ and decreasing for $s < 0$.
\end{itemize}

The potential $V$ will be assumed to satisfy a symmetry condition which we explain in the next section.

\subsection{Statement of the main result}
\label{section1.1}

\hspace{.5cm} Let $1 \leq k \leq N-1$ be an integer which determine the dimension of the sphere in which the solutions obtained are going to concentrate. Consider $\mathcal{H}$ a $(N-k-1)-$dimensional linear subspace of $\mathbb{R}^N$ and note that $\mathcal{H}^\perp$ is a $(k+1)-$dimensional subspace. All along the paper we use the notation for $x\in \mathbb{R}^N$ as $x = (x',x'')$ in which $x'\in \mathcal{H}$, $x''\in \mathcal{H}^\perp$ are such that $x = x' + x''$.

From now on, if $h:\mathbb{R}^N \to \mathbb{R}$ is a function, by saying that $h(x',x'') = h(x',|x''|)$ (which rigorously does not make sense), we mean that $h(x',y) = h(x',z)$ for all $y,z \in \mathcal{H}^\perp$ such that $|y| = |z|$. 

The condition in $V$ which is considered is the following:

\begin{itemize}
\item [$(V_1)$] There exists $V_0 > 0$ such that $V_0 \leq V(x)$ and, for all $x \in \mathbb{R}^N$, $V(x) = V(x',x'') = V(x',|x''|)$.
\end{itemize}

Unlike spike layered solutions, whose concentration occurs around minimum points of $V$, the solutions we are going to study concentrate around minimum points of an auxiliary potential. To see how we define it, let us consider the limit problem
\begin{equation}
-\Delta u + a u= f(u) \quad \mbox{in $\mathbb{R}^{N-k}$}.
\label{PL}
\end{equation}
It is well known (see \cite{Rabinowitz} for instance) that there exists a ground state solution $w\in H^1(\mathbb{R}^{N-k})$ of $(\ref{PL})$ which minimizes the energy functional
$$I_{a}(u) = \frac{1}{2}\int_{\mathbb{R}^{N-k}}\left(|\nabla u|^2 + au^2 \right)dx - \int_{\mathbb{R}^{N-k}}F(u)dx,$$
in the corresponding Nehari manifold given by
$$\mathcal{N}_{a} = \{u \in H^1(\mathbb{R}^{N-k})\backslash\{0\}; I_{a}'(u)u = 0\}.$$

We define the {\it ground-energy function} $\mathcal{E}: \mathbb{R}^+ \to \mathbb{R}^+$, by
$$\mathcal{E}(a) = \inf_{\mathcal{N}_{a}}I_{a}.$$
Finally, we define the {\it auxiliary potential} $\mathcal{M}:\mathbb{R}^N \to (0,+\infty]$ by
$$\mathcal{M}(x) = |x''|^k \mathcal{E}(V(x))$$ where $x = (x',x'')$, $x'\in \mathcal{H}$ and $x''\in \mathcal{H}^\perp$.

On the auxiliary potential $\mathcal{M}$ we impose the following condition
\begin{itemize}
\item [$(\mathcal{M}_1)$] There exists an open bounded set $\Omega \subset \mathbb{R}^N$ such that, if $(x',x'') \in \Omega$ then $(x',y'') \in \Omega$ for all $y''\in \mathcal{H}^\perp$, $|x''|=|y''|$. Moreover

$$0 < \mathcal{M}_0 : = \inf_{x \in \Omega} \mathcal{M}(x) < \inf_{x\in \partial \Omega} \mathcal{M}(x).$$
\end{itemize}

Now we can finally state our main result.

\begin{theorem}
Let $f$ satisfying $(f_1) - (f_4)$ and $V$ such that $(V_1)$ and $(\mathcal{M}_1)$ hold. Then for each sequence $\epsilon_n \to 0$, there exists a subsequence still denoted by $(\epsilon_n)$ such that (\ref{P}) (with $\epsilon = \epsilon_n$) has a nodal bound state $u_n$ such that, $u(x',x'') = u(x',|x''|)$ and, if $\epsilon_nP^1_n$ and $\epsilon_nP^2_n$ are respectively a minimum and a maximum point of $u_n$, then $\epsilon_nP^i_n \in \Omega$, $i = 1,2$ for $n$ sufficiently large, 
\begin{equation}
\epsilon_nP_n^i \to x_0, \quad \mbox{as $n \to \infty$}
\label{Pconcentration}
\end{equation}
where $\mathcal{M}(x_0) = \mathcal{M}_0$ and
$$|u_n(x)| \leq C\left(e^{-\frac{\beta}{\epsilon_n} d_k(x,\epsilon_n P_n^1)} + e^{-\frac{\beta}{\epsilon_n} d_k(x,\epsilon_nP_n^2)}\right) \quad x\in \mathbb{R}^N,$$
where $C, \beta > 0$ and $d_k$ is the distance defined in (\ref{distancedk}).
\label{theorem1.1}
\end{theorem}

The arguments in proving the existence of solutions were strongly influenced by the works of Alves and Souto \cite{AlvesSouto}, in which they prove the existence of nodal solutions to a Schr\"odinger-Poisson system. In the concentration, we follow closely the arguments in \cite{Soares1} and \cite{Bonheure1,Bonheure2}.

After our work has been finished we found the very recent paper \cite{Fei} in which the author uses a similar argumentation in order to prove the existence of a sequence of nodal multi-peak solutions which concentrate around the minimum points of a modified potential, associated to a vanishing potential. The existence arguments in both works rely on a minimization of the penalized energy functional on the nodal Nehari set and the concentration arguments follow the same general lines. Nevertheless, it is worth pointing out that in our work, since we get sphere concentrating solutions, several technicals difficulties arise. Moreover, in our work proving that the solution of the modified problem is in fact a solution of the original one involves different comparison functions, since our penalization is slightly different.

In Section 2 we present the penalization scheme and the variational framework. In Section 3 we prove the existence of the nodal solutions of the modified problem. In Section 4 we exhibit the concentration arguments in order to prove that the solutions of the modified problem concentrates around a k-dimensional sphere and in the last section we complete the prove of the Theorem \ref{theorem1.1} by showing that the solutions of the modified problem satisfy the original one.

\section{The penalized nonlinearity and the variational framework}

\hspace{.5cm}
The penalization we are going to apply is a variation of the classical method of Del Pino and Felmer in \cite{DelPino}, developed by Y. Sato in \cite{Sato} in order to allow its use in finding nodal solutions. Fixing $2 < \tau < \theta$, let $r_\epsilon > 0$ such that 
$$\frac{f(r_\epsilon)}{r_\epsilon} = \epsilon^\tau \quad \mbox{and} \quad \frac{f(-r_\epsilon)}{-r_\epsilon} = \epsilon^\tau.$$

Since $r_\epsilon \to 0$ as $\epsilon \to 0$, $(f_1)$ implies that
$$\epsilon^\tau = \frac{f(|r_\epsilon|)}{|r_\epsilon|} \leq |r_\epsilon|^{\nu-1}.$$
Thus $\epsilon^\frac{\tau}{\nu-1} \leq |r_\epsilon|$ and we can choose an odd function $\tilde{f}_\epsilon \in C^1(\mathbb{R})$ satisfying
$$\tilde{f}_{\epsilon}(s) = \left\{
\begin{array}{rl}
f(s) & \mbox{if \,$|s| \leq \frac{1}{2}\epsilon^\frac{\tau}{\nu-1}$,}\\
\epsilon^\tau s & \mbox{if \,$|s| \geq \epsilon^\frac{\tau}{\nu-1}$,}
\end{array}\right.$$
\begin{equation}
|\tilde{f}_{\epsilon}(s)| \leq \epsilon^\tau|s| \quad \mbox{for all $s \in \mathbb{R}$,}
\label{penalization1}
\end{equation}
\begin{equation}
0 \leq \tilde{f}_{\epsilon}'(s) \leq 2\epsilon^\tau \quad \mbox{for all $s \in \mathbb{R}$}
\label{penalization2}
\end{equation}
and
\begin{equation}
\mbox{ $s \mapsto \tilde{f}_\epsilon(s)/s$ is increasing for $s > 0$ and decreasing for $s < 0$.}
\label{penalization3}
\end{equation}

Let us define $g_\epsilon(x,s): = \chi_\Omega(x)f(s) + (1 - \chi_\Omega(x))\tilde{f}_{\epsilon}(s)$,
where $\chi_\Omega$ is the characteristic function of $\Omega$. Note that by $(f_1) - (f_4)$, $g$ is a Charath\'eodory function, such that $g_\epsilon(x',x'',s) = g_\epsilon(x',|x''|,s)$ satisfying 
\begin{itemize}
\item [$(g_1)$] $g_\epsilon(x,s) = o(|s|^\nu)$, as $s \to 0$, uniformly in compact sets of $\mathbb{R}^N$.
\item [$(g_2)$] There exist $c_ 1, c_2 > 0$ such that $ |g_\epsilon(x,s)| \leq c_1|s| + c_2|s|^p$ where $1 < p < \frac{N+2}{N-2}$;
\item [$(g_3)$] There exists $\theta > 2$ such that:
\begin{itemize}
\item [$i)$] $0 < \theta G_\epsilon(x,s) \leq g_\epsilon(x,s)s$, for $x \in \Omega$ and $s \neq 0$,
\item [$ii)$] $0 < 2 G_\epsilon(x,s) \leq g_\epsilon(x,s)s,$ for $x\in \mathbb{R}^N\backslash\Omega$ and $s \neq 0$,
\end{itemize}
where $G_\epsilon(x,s) = \int_0^s g_\epsilon(x,t)dt$.
\item[$(g_4)$] $s \mapsto \frac{g_\epsilon(x,s)}{s}$ is a nondecreasing function for $s > 0$ and nonincreasing for $s < 0$, for all $x\in \mathbb{R}^N$.
\end{itemize}

In a first moment, the concern will be with the penalized problem
\begin{equation}
-\epsilon^2\Delta u + V(x)u = g_\epsilon(x,u) \ \ \mbox{in
$\mathbb{R}^N$}. \label{P1}
\end{equation}
Taking $v_\epsilon(x) = u_\epsilon(\epsilon x)$, we relate each solution $u_\epsilon$ of (\ref{P1}) with a solution $v_\epsilon$ of
\begin{equation}
-\Delta v + V(\epsilon x)v = g_\epsilon(\epsilon x,u) \ \ \mbox{in
$\mathbb{R}^N$}. \label{P2}
\end{equation}

In order to obtain solutions of $(\ref{P2})$ with a partial symmetry, let us consider the following subspace of $H^1(\mathbb{R}^N)$,
$$\tilde{H} := \left\{v \in H^1(\mathbb{R}^N); \int (|\nabla v|^2 + V(\epsilon x)v^2) < +\infty \, \,  \mbox{and} \, \,  v(x',x'') = v(x',|x''|)\right\}$$
which is a Hilbert space when endowed with the inner product 
$$\langle u,v \rangle_\epsilon = \left( \int (\nabla u \nabla v + V(\epsilon x)uv) \right),$$
which gives rise to the following norm
$$\|v\|_\epsilon= \left( \int (|\nabla v|^2 + V(\epsilon x)v^2 \right)^\frac{1}{2}.$$

Since the approach is variational, let us consider the energy functional $I_\epsilon: \tilde{H} \to \mathbb{R}$, whose Euler-Lagrange equation is (\ref{P2}), given by
$$I_\epsilon(v) = \frac{1}{2}\int( |\nabla v|^2 + V(\epsilon x)v^2) - \int G_\epsilon(\epsilon x,v).$$
By standard arguments, one can prove that $I_\epsilon \in C^2(\tilde{H}, \mathbb{R})$.

\begin{remark}
In this section and through the rest of the paper, we omit the $dx$ in all the integrals and, when the domain over which the integral is calculated is $\mathbb{R}^N$ we write $\int$ rather than $\int_{\mathbb{R}^N}$.
\end{remark}
\section{Existence results}

\hspace{.5cm} Let us consider the Nehari manifold associated to (\ref{P2}), which is well defined by $(g_4)$ and given by
$$\mathcal{N}_\epsilon = \{v \in \tilde{H} \backslash \{0\}; I_\epsilon'(v)v = 0\}.$$

Since we are looking for nodal solutions, let us consider the so called nodal Nehari set
$$\mathcal{N}_\epsilon^\pm = \{v \in \tilde{H}; v^\pm \neq 0 \ \ \mbox{and} \ \ I_\epsilon'(v)v^\pm = 0\}.$$
Although $\mathcal{N}_\epsilon^\pm$ is not a manifold since $u \mapsto u^+$ in $H^1(\mathbb{R}^N)$ lacks differentiability, it is a set which contain all nodal solutions of (\ref{P2}).

Next result try to infer informations of $I_\epsilon$ with respect to $\mathcal{N}_\epsilon^\pm$ in the same way that one is used to do with $\mathcal{N}_\epsilon$.

\begin{lemma}
Let $v \in \tilde{H}$ such that $v^\pm \neq 0$. Then there exist $t,s > 0$ such that $t v^+ + s v^- \in \mathcal{N}_\epsilon^\pm$.
\label{lemma1}
\end{lemma}
\begin{proof}
First of all let us prove that, for all $v \in \tilde{H}\backslash\{0\}$, there exists $t > 0$ such that $tv \in \mathcal{N}_\epsilon$. Indeed, if $\Gamma = (supp v \cap \Omega_\epsilon) \cup \{x \in \mathbb{R}^N; \, |v(x)|\leq \frac{1}{2}\epsilon^\frac{\tau}{\nu - 1}\}$, note that $|\Gamma| > 0$ (since $v \in H^1(\mathbb{R}^N)$) and then
\begin{eqnarray*}
I_\epsilon(tv)& = & \frac{t^2}{2}\|v\|^2_\epsilon - \int_\Gamma F(tv) - \int_{\mathbb{R}^n\backslash\Gamma} G(\epsilon x,v)\\
& \leq & \frac{t^2}{2}\|v\|^2_\epsilon - t^\theta \int_\Gamma |v|^\theta - \frac{t^\tau}{2}\int_{\mathbb{R}^n\backslash\Gamma} |v|\\
& \to & -\infty,
\end{eqnarray*}
as $t \to \infty$. Then, if $v \in \tilde{H}$ is such that $v^\pm \neq 0$, there exist $t,s > 0$ such that
$$I'_\epsilon(tv^+)tv^+ = 0 \quad \mbox{and} \quad I'_\epsilon(sv^-)sv^- = 0.$$
Then, it is clear that 
$$I'_\epsilon(tv^+ + sv^-)(tv^+ + sv^-) = I'_\epsilon(tv^+)tv^+ + I'_\epsilon(sv^-)sv^- = 0.$$
\end{proof}

For a fixed $v \in \tilde{H}$, let us consider $\psi_v:[0, +\infty)\times [0,+\infty) \to \mathbb{R}$ given by
$$\psi_v(t,s) = I_\epsilon(t v^+ + s v^-),$$
and note that by the smoothness of $g$, $\psi_v \in C^2(\mathbb{R}^2, \mathbb{R})$.

\begin{lemma}
Let $v \in \mathcal{N}_\epsilon^\pm$, then $(t,s) = (1,1)$ is a strict global maximum point of $\psi_v$.
\label{lemma2}
\end{lemma}
\begin{proof}
First of all let us note that if $v \in \mathcal{N}_\epsilon^\pm$, by $(g_3)$
$$\lim_{|(t,s)| \to \infty}\psi_v(t,s) = -\infty.$$
Then there exists $R>0$ such that 
\begin{equation}
\psi_v(t,s) < 0, \quad \mbox{if $|(t,s)| \geq R.$}
\label{eq1}
\end{equation}
Since $$\nabla\psi_v(t,s) = \left(I_\epsilon'(tv^+)v^+, I_\epsilon'(sv^-)v^-\right),$$ standard calculations about the behavior of $t \mapsto I_\epsilon(tv^+)$ and $s \mapsto I_\epsilon(sv^-)$ and the fact that $v \in \mathcal{N}_\epsilon^\pm$, implies that $\psi_v$ has just one critical point given by $(t,s) = (1,1)$.

As we prove in Lemma \ref{lemma4} (which is totally independent of this one), $\psi_v(1,1) = I_\epsilon(v) \geq \rho > 0$. By (\ref{eq1}) in order to get the result its enough to prove that $(1,1)$ is a local maximum point of $\psi_v$.
Note that
$$D^2\psi_v(t,s)=
\left(
\begin{array}{cc}
I_\epsilon''(tv^+)(v^+,v^+) & 0\\
 \\
0 & I_\epsilon''(sv^-)(v^-,v^-)
\end{array} \right)$$
and then 
\begin{eqnarray*}
\det(D^2\psi_v(1,1)) & = & I_\epsilon''(v^+)(v^+,v^+) . I_\epsilon''(v^-)(v^-,v^-)\\
& =  & \left(\int \left(g_\epsilon(\epsilon x,v^+)v^+ - g_\epsilon'(\epsilon x,v^+){v^+}^2\right)\right)\\
& &  . \left(\int \left(g_\epsilon(\epsilon x,v^-)v^- - g_\epsilon'(\epsilon x,v^-){v^-}^2\right)\right).
\end{eqnarray*}
By definition of $g_\epsilon$ and $(g_4)$, the last integral is greater or equal to
$$
\begin{array}{c}
\displaystyle \left(\int_{(supp(v^+)\cap\Omega_\epsilon) \cup \{|v^+| \leq \frac{1}{2}\epsilon^\frac{\tau}{\nu-1}\})} \left(f(v^+)v^+ - f'(v^+){v^+}^2\right)\right) \\
. \displaystyle
\left(\int_{supp(v^-)\cap(\Omega_\epsilon \cup \{|v^-| < \frac{1}{2}\epsilon^\frac{\tau}{\nu-1}\})} \left(f(v^-)v^- - f'(x,v^-){v^-}^2\right)\right)\\
> 0.
 \end{array}$$
In the last inequality we have used that by $(f_4)$, 
$$f(s)s - f'(s)s^2 < 0, \quad \mbox{for all $s\neq 0$}$$
and as $v^\pm \in \tilde{H}$, $|supp(v^+)\cap(\Omega_\epsilon \cup \{|v^+| < a\})| > 0$ and $|supp(v^-)\cap(\Omega_\epsilon \cup \{|v^-| < a\}| > 0$, where $\Omega_\epsilon: = \epsilon^{-1}\Omega$.

Since $D^2\psi_v(1,1)$ is a positive definite form, we have just to verify that $\frac{\partial^2 \psi_v}{\partial t^2} = I_\epsilon''(tv^+){v^+}^2 < 0$. But this follows since $1$ is a maximum point of $t \mapsto I_\epsilon(tv^+)$.
\end{proof}

Still as a consequence of the arguments employed in the construction of the Nehari manifold as in \cite{Rabinowitz}, it follows the following result.

\begin{lemma}
Let $v \in \tilde{H}$ such that $v^\pm \neq 0$ and
$$I_\epsilon'(v)v^\pm \leq 0,$$
then there exists $t,s \in (0,1]$ such that
$$tv^+ + sv^- \in \mathcal{N}_\epsilon^\pm.$$
\label{lemma3}
\end{lemma}
\begin{proof}
In fact, let $t \in \mathbb{R}$ such that  $I_\epsilon'(tv^+)tv^+ = 0$. Suppose by contradiction that $t > 1$, then
\begin{eqnarray*}
\|v^+\|^2 & = & \int \frac{g_\epsilon(\epsilon x,tv^+)v^+}{t} = \int_{v^+ > 0} \frac{g_\epsilon(\epsilon x,tv^+){v^+}^2}{tv^+}\\
& = & \int_{supp(v^+)\cap(\Omega_\epsilon\cup\{|tv^+| < \frac{1}{2}\epsilon^\frac{\tau}{\nu-1}\}}\frac{f(tv^+){tv^+}^2}{v^+} + \\
& & \int_{supp(v^+)\cap(\Omega_\epsilon\cup\{|tv^+| \geq  \frac{1}{2}\epsilon^\frac{\tau}{\nu-1}\}\}}\frac{g_\epsilon(\epsilon x,tv^+){v^+}^2}{tv^+} \\
& > & \int_{supp(v^+)\cap(\Omega_\epsilon\cup\{|tv^+| < \frac{1}{2}\epsilon^\frac{\tau}{\nu-1}\}}f(v^+)v^+\\
& & + \int_{supp(v^+)\cap(\Omega_\epsilon\cup\{|v^+| \geq  \frac{1}{2}\epsilon^\frac{\tau}{\nu-1}\}\}}\frac{g_\epsilon(\epsilon x,v^+){v^+}^2}{v^+} \\
& = & \int g_\epsilon(\epsilon x,v^+)v^+,
\end{eqnarray*}
which implies that $I_\epsilon'(v)v^+ > 0$, contradicting the hypothesis.
The same argument applies to $v^-$ and $s$.
\end{proof}

Let us define
$$d_\epsilon := \inf_{\mathcal{N}_\epsilon^\pm} I_\epsilon,$$ 
and note that if there exists a solution of (\ref{P2}) in the energy level $d_\epsilon$, then it is the solution with least energy among all nodal ones.

Now we are going to state and prove the main result of this section.

\begin{theorem}
For sufficiently small $\epsilon > 0$, there exists a nodal solution of (\ref{P2}), $v_\epsilon \in \tilde{H}$ such that $I_\epsilon(v_\epsilon) = d_\epsilon$.
\label{theorem2}
\end{theorem}

Before proceed with the proof of Theorem \ref{theorem2} let us state some technical result about $\mathcal{N}_\epsilon^\pm$.

\begin{lemma}
It holds that
\begin{description}
\item [$i)$] There exists $\rho > 0$ such that $\|v\|_\epsilon \geq \rho$ for all $v \in \mathcal{N}_\epsilon$.
\item [$ii)$] There exists a constant $C > 0$ such that, for all $v \in \mathcal{N}_\epsilon$, $\displaystyle I_\epsilon(v) \geq C\|v\|_\epsilon^2$.
\end{description}
\label{lemma4}

\end{lemma}
\begin{proof}
The proof of $i)$ follows by standard arguments. Let us prove just $ii)$, which in fact also follows by very known arguments. Note that, if $v \in \mathcal{N}_\epsilon$, by $(g_3)$
\begin{eqnarray*}
I_\epsilon(v) = I_\epsilon(v) -\frac{1}{\theta}{I'}_\epsilon(v)v & = & \left( \frac{1}{2} - \frac{1}{\theta}\right) \parallel v\parallel_\epsilon^2 +\frac{1}{\theta} \int_{{\mathbb{R}}^N} \left( g(x,v)v - \theta G(x,v)\right)dx\\
& \geq & \left( \frac{1}{2} - \frac{1}{\theta}\right) \parallel v\parallel_\epsilon^2 +\frac{1}{\theta} \int_{\Omega_\epsilon^c} \left( g(\epsilon x,v_n)v - \theta G(\epsilon x,v)\right)dx\\
& \geq & \left( \frac{1}{2} - \frac{1}{\theta}\right) \parallel v\parallel_\epsilon^2 +\frac{(2-\theta)}{\theta}\int_{\Omega_\epsilon^c} G(x,u_n)dx\\
& \geq & \left( \frac{1}{2} - \frac{1}{\theta}\right) \parallel v\parallel_\epsilon^2 + \frac{(2-\theta)}{2\theta}\int_{ \Omega_\epsilon^c} \epsilon^\tau v^2 dx\\
& = & \left(\frac{\theta - 2}{2\theta}\right)\int_{\mathbb{R}^N} \left(|\nabla v|^2 +(V(\epsilon x)-\epsilon^\tau) v^2\right)dx\\
& = & \left(\frac{\theta - 2}{2\theta}\right)\int_{\mathbb{R}^N} \left(|\nabla v|^2 +(V_0-\epsilon^\tau) v^2\right)dx\\
& \geq &C \parallel v\parallel_\epsilon^2,                   
\end{eqnarray*}
where $C > 0$ for $\epsilon > 0$ sufficiently small.
\end{proof}

\begin{proof}[{\bf Proof of Theorem \ref{theorem2}}]
The proof will be carried out into two steps.
In the first one we prove that $d_\epsilon$ is attained by a function $u_\epsilon \in \tilde{H}$.

Let $(w_n)$ be a minimizing sequence for  $I_\epsilon$ in $\mathcal{N}_\epsilon^\pm$, i.e., a sequence $(w_n) \subset \mathcal{N}_\epsilon^\pm$ such that
\begin{equation}
\lim_{n \to \infty} I_\epsilon(w_n) = d_\epsilon.
\label{eq2}
\end{equation}

Note that by (\ref{eq2}) and Lemma \ref{lemma4}, $(w_n)$ is a bounded sequence in $\tilde{H}$. Then there exists $w_\epsilon$ such that $w_n \rightharpoonup w_\epsilon$ in $\tilde{H}$ up to a subsequence. In the same way as in Lemma 2.3 in \cite{Castro}, it is possible to show that $v \mapsto v^\pm$ is a continuous function of $\tilde{H}$ into itself, from which follows that $w_n^\pm \rightharpoonup w_\epsilon^\pm$ in $\tilde{H}$. As a consequence, up to a subsequence
\begin{equation}
w_n^\pm \to w_\epsilon^\pm, \quad \mbox{a.e. in $\mathbb{R}^N$}.
\label{eq5}
\end{equation}
and
\begin{equation}
w_n^\pm \to w_\epsilon^\pm, \quad \mbox{in $L^r(\Omega_\epsilon)$, for $1 \leq r < \frac{2(N-k)}{N-k-2}$}
\label{eq5.1}
\end{equation}
where $w_\epsilon^\pm \neq 0$ by the same arguments in Lemma 2.8 in \cite{Sato}.

Since $w_\epsilon^\pm \neq 0$, let $t_\epsilon, s_\epsilon > 0$ be such that $t_\epsilon w_\epsilon^+ + s_\epsilon w_\epsilon^- \in \mathcal{N}_\epsilon^\pm$. By weak lower-semicontinously of $\|\cdot\|$ and by Sobolev embeddings it follows that
\begin{equation}
\|t_\epsilon w_\epsilon^+ + s_\epsilon w_\epsilon^-\|_\epsilon \leq \liminf_{n \to \infty} \|t_\epsilon w_n^+ + s_\epsilon w_n^-\|_\epsilon
\label{eq5.2}
\end{equation}
and 
\begin{equation}
\int_{\Omega_\epsilon} F(t_\epsilon w_\epsilon^+ + s_\epsilon w_\epsilon^-) = \liminf_{n \to \infty} \int_{\Omega_\epsilon} F(t_\epsilon w_n^+ + s_\epsilon w_n^-).
\label{eq5.3}
\end{equation}

Now the real essence of the modification of penalization really comes up. Note that for $\epsilon > 0$ sufficiently small
$$I_ {\epsilon, \mathbb{R}^N\backslash \Omega_\epsilon}(v):= \frac{1}{2}\int_{\mathbb{R}^N\backslash \Omega_\epsilon}\left(|\nabla v|^2 + V(\epsilon x)v^2\right) - \int_{\mathbb{R}^N\backslash \Omega_\epsilon}\tilde{F}_{\epsilon}(v)$$
is a strictly convex functional in $\tilde{H}({\mathbb{R}^n\backslash \Omega_\epsilon}) = \left\{ v \in H^1(\mathbb{R}^N\backslash\Omega_\epsilon), \int_{\mathbb{R}^N\backslash\Omega_\epsilon}\left(|\nabla v|^2 + V(\epsilon x)v^2\right) < \infty \right\}$. In fact, for $v,h \in \tilde{H}({\mathbb{R}^n\backslash \Omega_\epsilon})$, $h \neq 0$, by (\ref{penalization2}) and Sobolev embeddings
\begin{eqnarray*}
I_ {\epsilon, \mathbb{R}^N\backslash \Omega_\epsilon}''(v)(h,h) & = & \|h\|_{\epsilon,\mathbb{R}^n\backslash \Omega_\epsilon}^2 - \int_{\mathbb{R}^n\backslash \Omega_\epsilon}\tilde{f}_{\epsilon}'(v)h^2\\
& \geq & \|h\|_{\epsilon,\mathbb{R}^n\backslash \Omega_\epsilon}^2\left(1 - \epsilon^\tau\right) > 0,
\end{eqnarray*}
for $\epsilon > 0$ sufficiently small.
Then, by convex analysis it follows that $I_ {\epsilon, \mathbb{R}^N\backslash \Omega_\epsilon}$ is weakly lower semicontinuous. Then (\ref{eq5.2}), (\ref{eq5.3}) and this fact imply that
\begin{eqnarray*}
I_\epsilon(t_\epsilon w_\epsilon^+ + s_\epsilon w_\epsilon^-) & \leq & \liminf_{n \to \infty} I_\epsilon(t_\epsilon w_\epsilon^+ + s_\epsilon w_\epsilon^-)\\
& \leq & \liminf_{n \to \infty} (I_\epsilon(t_\epsilon w_\epsilon^+) + I_\epsilon(s_\epsilon w_\epsilon^-))\\
& \leq & \liminf_{n \to \infty} (I_\epsilon(w_\epsilon^+) + I_\epsilon(w_\epsilon^-))\\
& = & b_\epsilon.
\end{eqnarray*}

Hence $I_\epsilon(t_\epsilon w_\epsilon^+ + s_\epsilon w_\epsilon^-) = b_\epsilon$.

The second step is proving that $v_\epsilon \in \tilde{H}$ which minimizes $I_\epsilon$ on $\mathcal{N}_\epsilon^\pm$ is a critical point of $I_\epsilon$ in $\tilde{H}$. This can be done by employing the same arguments of Section 3 in \cite{Bartsch}. For the sake of completeness we include all the details of this proof.

Supposing by contradiction that $I_\epsilon'(v_\epsilon) \neq 0$, there exists $\delta, \lambda > 0$ such that 
\begin{equation}
\| I_\epsilon'(v)\|_* \geq \lambda, \quad \mbox{for all $v \in B_\delta(v_\epsilon) \subset \tilde{H}$.}
\label{eq9}
\end{equation}

Let us consider the function $(t,s) \mapsto tv_\epsilon^+ + sv_\epsilon^-$ defined on $D = \left( \frac{1}{2}, \frac{3}{2} \right)$ and note that by Lemma \ref{lemma2}
\begin{equation}
\delta_\epsilon:= \max_{(t,s) \in \partial D} I_\epsilon( tv_\epsilon^+ + sv_\epsilon^-) < I_\epsilon(v_\epsilon) = d_\epsilon.
\label{eq10}
\end{equation}
Taking $\rho = \min\left\{\frac{d_\epsilon - \delta_\epsilon}{2}, \frac{\lambda\delta}{24}\right\}$ and $S = B_{\delta/3}(v_\epsilon)$, Lemma 2.3 in \cite{Willem} yields a deformation $\eta$ such that
\begin{description}
\item [$i)$] $\eta(1,v) = v$ for all $v \not\in I_\epsilon^{-1}([d_\epsilon-2\rho,d_\epsilon+2\rho])$,
\item [$ii)$] $\eta(1,I_\epsilon^{d_\epsilon+\rho}\cap S) \subset I_\epsilon^{d_\epsilon-\rho}$,
\item [$iii)$] $I_\epsilon(\eta(1,v)) \leq I_\epsilon(v)$, for all $u \in \tilde{H}$.
\end{description}
\begin{claim}
$$\max_{(t,s)\in \overline{D}}I_\epsilon(\eta(1,tv_\epsilon^+ + sv_\epsilon^-)) < d_\epsilon.$$
\end{claim}

In fact, if $(t,s)\neq (1,1)$, then by Lemma \ref{lemma2},
$$I_\epsilon(\eta(1,tv_\epsilon^+ + sv_\epsilon^-)) \leq I_\epsilon(tv_\epsilon^+ + sv_\epsilon^-) < d_\epsilon.$$
On the other hand, by $ii)$,
$$I_\epsilon(\eta(1,v_\epsilon)) < d_\epsilon - \rho,$$
which proves the claim.

Now, let us prove that there exists $(t,s) \in D$ such that $\eta(1,tv_\epsilon^+ + sv_\epsilon^-) \cap \mathcal{N}_\epsilon^\pm \neq \emptyset$, which together with the Claim, contradicts the definition of $d_\epsilon$.

Define $h(t,s) = \eta(1,tv_\epsilon^+ + sv_\epsilon^-)$ and $\Psi_0, \Psi_1: \mathbb{R}^2 \to \mathbb{R}^2$ by
$$\Psi_0(t,s) = (I_\epsilon'(t v_\epsilon^+)v_\epsilon^+, I_\epsilon'(s v_\epsilon^-)v_\epsilon^-),$$
and
$$\Psi_1(t,s) = \left(\frac{1}{t}I_\epsilon'(h(t,s)^+)h(t,s)^+, \frac{1}{s}I_\epsilon'(h(t,s)^-)h(t,s)^-)\right).$$
By results of Brower Degree Theory, $\deg(\Psi_0,D,(0,0)) = 1$. On the other hand, note that by (\ref{eq10}), $(t,s)\mapsto tv_\epsilon^+ + sv_\epsilon^-$ coincides with $h$ on $\partial D$. Hence $\Psi_0 = \Psi_1$ on $\partial D$ and then $\deg(\Psi_1,D,(0,0)) = \deg(\Psi_0,D,(0,0)) = 1$. Therefore there exists $(t,s) \in D$ such that $\Psi_1(t,s) = (0,0)$ and consequently
$$\eta(1,tv_ \epsilon^+ + sv_\epsilon^-) = h(t,s) \in \mathcal{N}_\epsilon^\pm.$$

Finally, this contradiction proves the theorem.
\end{proof}

\section{Concentration results}

Let us introduce a sequence $\epsilon_n \to 0$ as $n \to \infty$ and, for each $n \in \mathbb{N}$, let us denote by $v_n$ the solution $v_{\epsilon_n}$ given by Theorem \ref{theorem2} and consider $d_n:= d_{\epsilon_n}$, $\|\cdot\|_n:=\|\cdot\|_{\epsilon_n}$ and $I_n: = I_{\epsilon_n}$.

The following result provides an upper estimate for the sequence of minimax values $d_n$. Its proof is inspired in the arguments of Alves and Soares in \cite{Soares1}.

\begin{lemma}
$$\limsup_{n \to \infty}\epsilon_n^k d_n \leq 2\omega_k\inf_{\Omega\cap \mathcal{H}^\perp} \mathcal{M}$$
and
$$\epsilon_n^k\|v_n\|_n^2 \leq C.$$
\label{lemma5.1}

\end{lemma}
\begin{proof}
Let $z_0=(x_0,y_0) \in \Omega$ be such that $\mathcal{M}(z_0) = \inf_{x \in \Omega} \mathcal{M}(x)$. Since $\Omega$ is an open set, there exists $R > 0$ such that $B_{2R}(z_0) \subset \Omega$. Let us choose points $z_1,z_2 \in \partial  B_R(z_0)$ such that, if $z_i=(Q_i',Q_i'')$, then $|Q_1'' - Q_2''| = 2R$. Note that $B_R(z_i) \subset \Omega$ for $i = 1,2$.
In the rest of this proof, $i \in \{1,2\}$.
Let us choose smooth cut-off functions $\psi_i:\mathbb{R}^{N-k} \to \mathbb{R}$ such that $\psi_i = 1$ in $B_{\mathbb{R}^{N-k}}((Q_i',|Q_i''|),R/4)$ and $\psi_i = 0$ in $\mathbb{R}^{N-k} \backslash B_{\mathbb{R}^{N-k}}((Q_i',|Q_i''|),R/2)$.

Let $w_i \in H^1(\mathbb{R}^{N-k})$, be a ground-state solution of
$$
-\Delta u + V(z_i) u= f(u) \quad \mbox{in $\mathbb{R}^{N-k}$}
$$
and $w_{n,z_i}:\mathbb{R}^N \to \mathbb{R}$ be given by
$$ w_{n,z_i}(x',x'') = \psi_i(\epsilon_n x',\epsilon_n |x''|)w_i \left(x' - \frac{Q_i'}{\epsilon_n},|x''| - \frac{|Q_i''|}{\epsilon_n}\right).$$
We associate for each $w_{n,z_i}(x)$ its $k-$dimensional counterpart $\tilde{w}_{n,z_i}(x',r) = w_{n,z_i}(x',x'')$ where $|x''| = r \in \mathbb{R}$.
It is well known the existence of $t_{i_n} > 0$ such that $t_{i_n}w_{n,z_i} \in \mathcal{N}_{\epsilon_n}$.
By the construction we have made so far, it is straightforward to see that 
$$\overline{w}_n := t_{1_n}w_{n,z_1} - t_{2_n}w_{n,z_2} \in \mathcal{N}_{\epsilon_n}^\pm.$$

Using the fact that $I_n$ is even, we have
\begin{equation}
d_n  \leq  I_n(\overline{w}_n) = I_n(t_{1_n}w_{n,z_1}  ) + I_n(t_{2_n}w_{n,z_2}).
\label{eq11}
\end{equation}

By a change of variable, let us note that for $i \in \{1,2\}$, $I_n(t_nw_{n,z_i})$ is equal to
$$
\begin{array}{l}
\displaystyle\frac{t_{i_n}^2}{2}\int_{\mathbb{R}^N}\left(|\nabla w_{n,z_i}|^2 + V(\epsilon_n x)w_{n,z_i}^2\right)dx - \int_{\mathbb{R}^N}G_n(\epsilon_n x, t_{i_n}w_{n,z_i})dx \\ \\
= \displaystyle\omega_k\left[ \frac{t_{i_n}^2}{2}\int_{\mathbb{R}^{N-k-1}}\int_0^{+\infty}r^k\left(|\nabla\tilde{w}_{n,z_i}(x',r)|^2 + V(\epsilon_n x', \epsilon_n r)\tilde{w}_{n,z_i}(x',r)^2\right)drdx'\right.\\ \\
\displaystyle\left. - \int_{\mathbb{R}^{N-k-1}}\int_0^{+\infty}r^kG_n(\epsilon_n x, \epsilon_n r,t_{i_n}\tilde{w}_{n,z_i}(x',r))drdx'\right] \\ \\
= \displaystyle\omega_k\left[ \frac{t_{i_n}^2}{2}\int_{\mathbb{R}^{N-k-1}} \int_{\frac{-|Q_i''|}{\epsilon_n}}^{+\infty}\left(\sigma + \frac{|Q_i''|}{\epsilon_n}\right)^k\left(|\nabla(\psi_i(\epsilon_n x' + Q_i',\epsilon_n\sigma+|Q_i''|)w_i(x',\sigma)) |^2\right. \right.\\ \\
\displaystyle\left. \left. + V(\epsilon_n x' + \epsilon_nQ_i',\epsilon_n\sigma + |Q_i''|)(\psi_i(\epsilon_n x' + \epsilon_n Q_i',\epsilon_n\sigma+|Q_i''|)w_i(x',\sigma)^2\right)d\sigma dx'\right.\\ \\
\displaystyle\left. - \hspace{-0.2cm} \int_{\frac{-|Q_i''|}{\epsilon_n}}^{+\infty}\hspace{-0.2cm}\left(\sigma \hspace{-0.1cm} + \hspace{-0.1cm} \frac{|Q_i''|}{\epsilon_n}\right)^k \hspace{-0.3cm}G_n(\epsilon_n x' \hspace{-0.2cm}+ \hspace{-0.1cm} \epsilon_n Q_i',\epsilon_n\sigma \hspace{-0.1cm} + \hspace{-0.1cm} |Q_i''|, t_{i_n}\psi_i(\epsilon_n x' \hspace{-0.1cm}+ \hspace{-0.1cm} \epsilon_n Q_i', \epsilon_n\sigma \hspace{-0.1cm}+ \hspace{-0.1cm}|Q_i''|)w_i(x',\sigma))d\sigma dx' \hspace{-0.1cm}\right] \\ \\
= \displaystyle\epsilon_n^{-k}\omega_k|Q_i|^k I_{V(z_i)}(t_{i_n}w_i) + o(1) \leq \epsilon_n^{-k}\omega_k |Q_i|^k \mathcal{E}(V(z_i)) = \epsilon_n^{-k}\omega_k\mathcal{M}(z_i) + o(1).
\displaystyle\end{array}
$$

By the last inequality and (\ref{eq11}) we have
\begin{equation}
\epsilon_n^k d_n \leq \omega_k\left(\mathcal{M}(z_1) + \mathcal{M}(z_2)\right) + o(1)
\label{eq12}
\end{equation}
and making $R \to 0$, continuity of $V$ and $\mathcal{M}$ implies in the result.

Now, in order to see that
$$\epsilon_n^k\|v_n\|_n^2 \leq C,$$
just observe that $v_n \in \mathcal{N}_{\epsilon_n}$ and use $ii)$ of Lemma \ref{lemma4}.
\end{proof}

The next lemma implies that solutions found in Theorem \ref{theorem2} do not vanish when $n \to \infty$.

\begin{lemma}
Let $P_n^1$, $P_n^2$ be local maximum and minimum points of $v_n$, respectively. Then $P_n^i \in \Omega_{\epsilon_n}:=\epsilon_n^{-1}\Omega$,
$$v_n(P_n^1) \geq a \quad \mbox{and} \quad  v_n(P_n^2) \leq -a,$$
where $a > 0$ is such that $ f(a)/a= V_0/2$.
\label{lemma5.2}
\end{lemma}
\begin{proof}
First of all let us prove that $P_n^i \in \Omega_{\epsilon_n}$. 
Suppose by contradiction that $P_n^1 \not \in \Omega_{\epsilon_n}$. Since $P_n^1$ is a local maximum point, it follows that $\Delta v_n(P_n^1) \leq 0$. By definition of $g_n$ we have
$$ V_0 v_n(P_n^1) \leq-\Delta v_n(P_n^1) + V(\epsilon_n P_n^1)v_n(P_n^1) = \tilde{f}_{\epsilon_n}(v_n(P_n^1)) \leq \epsilon_n^\tau v_n(P_n^1),$$
which is impossible for $\epsilon > 0$ sufficiently small.
A similar argument applies to $P_n^2$.

Since $P_n^1 \in \Omega_{\epsilon_n}$ and $v_n(P_n^1) > 0$, from the definition of $g_n$ we have
$$0 \geq \Delta v_n(P_n^1) = \left(V(\epsilon_n P_n^1) - \frac{f(v_n(P_n^1))}{v_n(P_n^1)}\right)v_n(P_n^1).$$
Supposing by contradiction that $v_n(P_n^1) < a$. By the choice of $a > 0$ and $(f_5)$ it follows that
\begin{equation}
V_0 \leq V(\epsilon_n P_n^1) \leq \frac{f(v_n(P_n^1))}{v_n(P_n^1)} \leq \frac{f(a)}{a} = \frac{V_0}{2},
\label{eq13}
\end{equation}
which is a contradiction.
Analogously we prove that $v_n(P_n^2) \leq -a$.
\end{proof}

By the last result, there exist $P^1, P^2 \in \Omega$ such that, along a subsequence
\begin{equation}
\lim_{n \to \infty} \epsilon_n P_n^i = P^i, \quad i \in \{1,2\}.
\label{Pi}
\end{equation}

The same argument of \cite{Soares1} with short modifications can be used to prove the following result.
\begin{lemma}
Using the same notation that in the last result, it follows that
\begin{equation}
\lim_{n \to \infty}|P_n^1 - P_n^2| = +\infty.
\label{eq14}
\end{equation}
\label{lemma5.3}
\end{lemma}

\begin{lemma}
Let $y_n = (y_n',y_n'') \subset \mathbb{R}^N$ be a sequence such that $\epsilon_ny_n \to (\bar{y}',\bar{y}'') \in \overline{\Omega}$ as $n \to \infty$. Denoting $\tilde{v}_n(x',r):= v_n(x', x'')$ where $|x''| = r$, let us define $\tilde{w}_n: \mathbb{R}^{N-k-1}\times [-|y_n|,+\infty) \to \mathbb{R}$ by
$$\tilde{w}_n(x',r):= \tilde{v}_n(x',r + |y_n''|).$$
Then there exists $\tilde{w} \in H^1(\mathbb{R}^{N-k})$ such that $\tilde{w}_n \to \tilde{w}$ in $C^2_{loc}(\mathbb{R}^{N-k})$ and $\tilde{w}$ satisfies the limit problem
\begin{equation}
-\Delta \tilde{w} + V(\bar{y}',|\bar{y}''|)\tilde{w} = \tilde{g}_n(x',r, \tilde{w}) \quad \mbox{in $\mathbb{R}^{N-k}$,}
\label{Plim}
\end{equation}
where $\tilde{g}_n(x',r,s) := \chi(x',r)f(s) + (1-\chi(x',r))\tilde{f}_{\epsilon_n}(s)$ and $\chi(x',r)=\lim_{n\to\infty}\chi_\Omega(\epsilon_n x' + \epsilon_ny_n',\epsilon_n r + \epsilon_n|y_n''|)$.
\label{lemma5.4}
\end{lemma}
\begin{proof}

The proof is analogous to \cite{Bonheure1}[Lemma 4.3] but we sketch it here for the sake of completeness. 

Note that $\tilde{w}_n$ satisfies the following problem
\begin{equation}
-\Delta\tilde{w}_n- \frac{k}{(r + |y_n''|)}\frac{\tilde{w}_n}{\partial r} + V\hspace{-0.1cm}(\epsilon_n x' + \epsilon_n y_n',\epsilon_n r+ \epsilon_n|y_n''|)\tilde{w}_n \hspace{-0.1cm} = \hspace{-0.1cm} g_n(\epsilon_n x' + \epsilon_n y_n',\hspace{-0.1cm} \epsilon_n r+ \epsilon_n|y_n''|, \hspace{-0.1cm}\tilde{w}_n),
\label{PRtrans}
\end{equation}
in $\mathbb{R}^{N-k-1}\times [-|y_n|,+\infty)$.

By Lemma \ref{lemma5.1} it follows that 
\begin{equation}
\int_{\mathbb{R}^{N-k-1}} \int_{-|y_n''|}^{+\infty}\left(|\nabla \tilde{w}_n|^2 + V(\epsilon_n x' + \epsilon_n y_n',\epsilon_n r + \epsilon_n|y_n''|)\tilde{w}_n^2\right)dr dx'\leq C,
\label{estimatelemma5.4}
\end{equation}
uniformly in $n$ and then, for some $\tilde{w}\in H^1(\mathbb{R}^{N-k})$,
\begin{equation}
\tilde{w_n} \rightharpoonup \tilde{w} \quad \mbox{in $H^1(\mathbb{R}^{N-k})$.}
\label{convergencelemma5.4}
\end{equation}

By choosing a sequence $R_n \to \infty$ such that $\epsilon_n R_n \to 0$ and considering a smooth cut-off function in $\mathbb{R}^{N-k}$, $\eta_R$ such that $0 \leq \eta_R \leq 1$, $\eta_R(z) = 1$ if $|z| \leq \frac{R}{2}$ and $\eta_R(z) = 0$ if $|z| > R$, and $\|\nabla \eta_R\|_\infty \leq \frac{C}{R}$, it can be proved using (\ref{estimatelemma5.4}) that $\overline{w}_n(z) := \eta_{R_n}(z)\tilde{w}_n(z)$ is bounded in $H^1(\mathbb{R}^{N-k})$, uniformly in $n$. 

Since $\overline{w}_n$ satisfies (\ref{PRtrans}) in $B(0,R_n)$, it follows by classical elliptic estimates that 
\begin{equation}
\|\overline{w}_n\|_{W^{2,q}(B(0,R))} \leq C
\label{bootstrap}
\end{equation}
for sufficiently large $n\in \mathbb{N}$, where $R > 0$ is fixed.

By (\ref{convergencelemma5.4}) and (\ref{bootstrap}) it follows that $\tilde{w}_n \to \tilde{w}$ in $C^2_{loc}(\mathbb{R}^{N-k})$ and that $\tilde{w}$ satisfies (\ref{Plim}).
\end{proof}

Since the concentration set is expected to be a sphere in $\mathbb{R}^N$, it is natural to introduce the distance between two $k-$dimensional spheres in $\mathbb{R}^N$, which gives rise to neighborhoods in which we want to estimate the mass of solutions. For $x=(x',x''),y=(y',y'') \in \mathbb{R}^N$, let
\begin{equation}
d_k(x,y) = \sqrt{(x'-y')^2 + (|x''| - |y''|)^2},
\label{distancedk}
\end{equation}
which corresponds to the distance between $k-$dimensional spheres centered at the origin, parallel to $\mathcal{H}^\perp$ and of radius $|x''|$ and $|y''|$, respectively. According to this distance, the balls are given by
$$B_k(x,r) = \{y \in \mathbb{R}^N; \, d_k(x,y) < r\}.$$

From now on, for $\Lambda \subset \mathbb{R}^N$ we denote 
$$I_{n,\Lambda}(v) = \frac{1}{2}\int_\Lambda \left(|\nabla v|^2 + V(\epsilon_n x)v^2\right)dx - \int_\Lambda G_n(\epsilon_n x,v)dx.$$

The following is the main step in proving the concentration result.

\begin{proposition}
Suppose all the assumptions of Theorem \ref{theorem1.1} to hold. Then,
\begin{itemize}
\item [$i)$] $\displaystyle \lim_{n \to \infty} \epsilon_n^kd_n = 2\omega_k\inf_\Omega \mathcal{M}$.\\
\item [$ii)$] $\displaystyle \lim_{n \to \infty} \mathcal{M}(\epsilon_n P_n^i) = \inf_\Omega \mathcal{M}$, $i \in \{1,2\}$.\\
\end{itemize}
\label{prop5.6}
\end{proposition}
\begin{proof}
Let us get started with $i)$.
By Lemma \ref{lemma5.1} it follows that
\begin{equation}
\limsup_{n \to \infty} \epsilon_n^k d_n \leq 2\omega_k\inf_\Omega \mathcal{M}.
\label{eq15}
\end{equation}

In order to prove that 
\begin{equation}
\liminf_{n \to \infty} \epsilon_n^k d_n \geq 2\omega_k\inf_\Omega \mathcal{M},
\label{eq15.1}
\end{equation}
we use some ideas of \cite{DelPino}[Lemma 2.2].

By Lemmas \ref{lemma5.2}  and \ref{lemma5.4} it follows that $\tilde{w}^i_n(x',r):= \tilde{v}_n(x',r + P_n^i) \to \tilde{w}_i$ in $C^2_{loc}(\mathbb{R}^{N-k})$, where $\tilde{w}_i \neq 0$ and satisfies (\ref{Plim}) with $(\bar{y}',\bar{y}'') = P^i$.

For each $R > 0$ and up to a subsequence in $n$, Lemma \ref{lemma5.4} with calculations similar to which have resulted in (\ref{eq12}) implies that
\begin{eqnarray*}
\epsilon_n^k I_{n,B_k(P^i,R)}(v_n) & = & \omega_k|{P^i}''|^k \left(\frac{1}{2}\int_{B_R(0)} \left(|\nabla \tilde{w}_i|^2 + V(P^i)\tilde{w}_i^2\right)dz \right. - \\
& & \left. \int_{B_R(0)} \tilde{G}(r,\tilde{w}_i)dz\right) + o_n(1)\\
& \geq & \omega_k|{P^i}''|^k \left(\frac{1}{2}\int_{B_R(0)} \left({|{\nabla \tilde{w}_i}|}^2 + V(P^i)\tilde{w}_i^2\right)dz - \right. \\
& & \left. \int_{B_R(0)} F(\tilde{w}_i)dz \right) + o_n(1).
\end{eqnarray*}
Since $\tilde{w}_i \in H^1(\mathbb{R}^{N-k})$, it follows that for a given $\eta > 0$, there exists $R > 0$ such that
$$\liminf_{n \to \infty}\epsilon_n^k I_{n,B_k(y_n,R)}(v_n) \geq \omega_k\mathcal{M}(P^i) - \eta.$$

Taking into account Lemma \ref{lemma5.3}, it follows that for $n$ large enough $B_k(P_n^1,R)$ and $B_k(P_n^2,R)$ are disjoint and then 
$$\epsilon_n^kI_{n}(v_n) = \epsilon_n^k I_{n,B_k(P_n^1,R)}(v_n) + \epsilon_n^k I_{n,B_k(P_n^2,R)}(v_n) + I_{n,\mathbb{R}^N\backslash B_{n,R}}(v_n),$$
where $B_{n,R}:=B_k(P_n^1,R) \cup B_k(P_n^2,R)$.

Now, calculations similar to (2.21) in \cite{Pimenta} implies that $\liminf_{n \to \infty}I_{n,\mathbb{R}^N\backslash B_{n,R}}(v_n) \geq -\eta$ for $R > 0$ sufficiently large and then 

\begin{eqnarray*}
\liminf_{n\to\infty} \epsilon_n^kI_n(v_n) & \geq & \omega_k \left(\mathcal{M}(P^1) + \mathcal{M}(P^2)\right) -3\eta\\
& \geq & 2\omega_k\inf_\Omega \mathcal{M} - 3\eta.
\end{eqnarray*}
Since the last inequality holds for all $\eta > 0$, (\ref{eq15.1}) is proved.

To prove $ii)$ let us suppose by contradiction that 
$$\mathcal{M}(P^i) > \inf_\Omega \mathcal{M},  \quad i \in \{1,2\}.$$
where $P^i$ is given by (\ref{Pi}).
Just by arguing as in the first item, one can see that
$$\liminf_{n\to\infty} \epsilon_n^kI_n(v_n)  \geq  \omega_k \left(\mathcal{M}(P^1) + \mathcal{M}(P^2)\right) > 2\omega_k\inf_\Omega \mathcal{M},$$
which contradicts the statement of Lemma \ref{lemma5.1}. This contradiction proves the proposition.
\end{proof}

\section{Proof of Theorem \ref{theorem1.1}} 

Let $v_n$ be as in the beginning of Section 4 and $u_n(x):=v_n(\epsilon_n^{-1}x)$.

\begin{proposition}
$$\lim_{n \to \infty}\|v_n\|_{L^\infty(\Omega_n\backslash\left(B_k(P_n^1,R) \cup B_k(P_n^2,R))\right)} = 0.$$
\label{prop5.1}
\end{proposition}
\begin{proof}
The prove follows by contradiction. Let us suppose that there exist $\eta > 0$ and sequences $R_n \to \infty$ and $(y_n) \subset \Omega_n\backslash\left(B_k(P_n^1,R_n) \cup B_k(P_n^1,R_n)\right)$ such that 
$$|v_n(y_n)| \geq \eta.$$
Since $\epsilon_n y_n \in \Omega$, it follows that $\epsilon_n y_n \to \bar{y} \in \Omega$ up to a subsequence.

Then following the arguments in Proposition \ref{prop5.6} it is possible to show that
$$\epsilon_n^k d_n = I_n(v_n) \geq 3\omega_k \inf_{\Omega}\mathcal{M},$$
which contradicts Lemma \ref{lemma5.1}.
\end{proof}

By standard elliptic regularity theory it is possible to show that $v_n \in C^2(\mathbb{R}^N)$. Then, by continuity, Proposition \ref{prop5.1} implies that 
\begin{equation}
\|v_n\|_{L^\infty{\partial(\left(B_k(P_n^1,R_n) \cup B_k(P_n^1,R_n)\right))}} = o_n(1)
\label{boundarydecay}
\end{equation}

In order to prove the exponential decay with respect to $\epsilon_n$ of the functions $u_\epsilon$, let us take as a comparison function
$$W(x) = C(e^{-\beta d_k(x,P_n^1)} + e^{-\beta d_k(x,P_n^2)}),$$
defined in $\mathbb{R}^N \backslash \left(B_k(P_n^1,R) \cup B_k(P_n^1,R)\right)$, where $C > 0$ is to be choosen.
For sufficiently small $\beta > 0$ (independent of $\epsilon_n$), it follows that 
$$\left(-\Delta + V(\epsilon_nx) - \frac{g_n(\epsilon_x,v_n)}{v_n}\right)(W \, \pm\,  v_n) \geq 0 \quad \mbox{in $\mathbb{R}^N \backslash \left(B_k(P_n^1,R) \cup B_k(P_n^1,R)\right)$.}
$$
Then, by (\ref{boundarydecay}), for $x \in \partial \left(B_k(P_n^1,R) \cup B_k(P_n^1,R)\right)$
$$W(x) \pm v_n(x) = C2e^{-\beta R} \pm v_n \geq 0$$ for a sufficiently large constant $C > 0$ which does not depend on $n$. Hence Maximum Principle applies and 
$$|v_n| \leq W(x), \quad \mbox{on $\mathbb{R}^N \backslash \left(B_k(P_n^1,R) \cup B_k(P_n^1,R)\right)$.}$$

Then there exists a sufficiently large constant $C > 0$ such that 
$$|v_n(x)| \leq C\left(e^{-\beta d_k(x,P_n^1)} + e^{-\beta d_k(x,P_n^2)}\right),$$
for $x \in \mathbb{R}^N \backslash \left(B_k(P_n^1,R) \cup B_k(P_n^1,R)\right)$, which implies that
\begin{equation}
\displaystyle |u_n(x)| \leq C\left(e^{-\frac{\beta}{\epsilon_n} d_k(x,\epsilon_n P_n^1)} + e^{-\frac{\beta}{\epsilon_n} d_k(x,\epsilon_nP_n^2)}\right)
\label{decayestimate}
\end{equation}
for $x \in \mathbb{R}^N \backslash \left(B_k(\epsilon_n P_n^1,\epsilon_nR) \cup B_k(\epsilon_n P_n^1,\epsilon_n R)\right)$. In particular it holds that 
$$\|u_n\|_{L^\infty(\mathbb{R}^N\backslash\Omega)} \leq C e^\frac{-\beta}{\epsilon_n}$$
which implies that $u_n$ satisfies the original problem. 
The proof of (\ref{Pconcentration}) follows by (\ref{Pi}) and item {\it ii)} of Proposition \ref{prop5.6}. Therefore the proof of Theorem \ref{theorem1.1} follows.

\vspace{1cm}
\noindent \textbf{Acknowledgment.}\  

We would like to express our gratitude to Prof. Antonio Su\'arez for some discussions on this subject.
\vspace{0.3cm}

\noindent \textbf{Competing interests.}\
 The authors declare that they have no competing interests.
\vspace{0.3cm}

\noindent \textbf{Author's contribution.}\
All the author wrote, review and approve the final manuscript.

\end{document}